\documentclass[11pt,a4paper]{amsart}

\usepackage{microtype}
\usepackage[margin=31mm]{geometry}
\usepackage{amsmath,amssymb,amsthm,mathtools}
\usepackage{enumitem}
\usepackage{array}
\usepackage{xcolor}
\usepackage{hyperref}
\newtheorem{theorem}{Theorem}[section]
\newtheorem{proposition}[theorem]{Proposition}
\newtheorem{lemma}[theorem]{Lemma}

\newtheorem*{theoremA}{Theorem}

\theoremstyle{definition}

\theoremstyle{remark}
\newtheorem{remark}[theorem]{Remark}
\numberwithin{equation}{section}
\newcommand{\PP}{\mathbb P}
\newcommand{\RR}{\mathbb R}
\newcommand{\EE}{\mathbb E}
\newcommand{\Prob}{\operatorname{Prob}}
\newcommand{\Law}{\operatorname{Law}}
\newcommand{\Cpl}{\operatorname{Cpl}}
\newcommand{\dd}{\,d}
\newcommand{\1}{\boldsymbol 1}

\def\NN{{\mathbb N}}

\def\GL{\operatorname{GL}}
\def\SL{\operatorname{SL}}

\newcommand{\norm}[1]{\left\Vert #1 \right\Vert}

\title[Log-H\"older continuity]{Log-H\"oder continuity at zero Lyapunov gap for finite-state Markov $GL(2)$-cocycles}
\author{El Hadji Yaya Tall}
\address{Centro de investigación en Matemática Pura y Aplicada (CIMPA), Universidad de Costa Rica. San José, Costa Rica.}
\email{el.tall@ucr.ac.cr}
\date{\today}

\begin{document}

\begin{abstract}
We prove that the extremal Lyapunov exponents of finite-state Markov $\mathrm{GL}(2,\mathbb R)$-cocycles are pointwise log-Hölder continuous, jointly in the cocycle matrices and the transition kernel, at every parameter $(A,P)$ satisfying $\lambda_+(A,P)=\lambda_-(A,P)$.
Perturbations are taken within a fixed transition graph. The main new ingredient is a Markov perpetuity estimate for the nonsplit triangular case, obtained through a martingale–coboundary decomposition. In the conformal case, the exponent $1/2$ in the logarithmic modulus can be replaced by $1$.
\end{abstract}
\maketitle
\tableofcontents

\section{Introduction and Main results}
\label{sec:intro}

Let $X=\{1,\ldots,N\}$, $M=X^{\NN}$ and denote by $\sigma:M\to M$ the left shift, $(\sigma x)_k=x_{k+1}$. For every \(k\in\NN\), let
$$
X_k:M\to X,
\qquad
X_k(x)=x_k,
$$
be the $k$-th coordinate projection. Thus, under a Markov measure,
$$(X_k)_{k\in\NN}$$ is the corresponding stationary Markov chain.

Given $i_0,\ldots,i_n\in X$, we write
$$
[0; i_0,\ldots,i_n]
=
\left\{
x\in M:
X_0(x)=i_0,\ldots,X_n(x)=i_n
\right\}
$$
for the associated cylinder set.

Let $P=(P_{ij})_{i,j\in X}$ be a primitive stochastic matrix. This
means that
$$
P_{ij}\geq 0,
\qquad
\sum_{j\in X}P_{ij}=1
\quad\text{for every }i\in X,
$$
and that there exists $m_0\geq 1$ such that
$$
(P^{m_0})_{ij}>0
\quad\text{for every }i,j\in X.
$$

Since $P$ is primitive, it admits a unique stationary probability
row vector
$$
p(P)=\bigl(p_1(P),\ldots,p_N(P)\bigr),
$$
characterized by
$$
p(P)P=p(P),
\qquad
p_i(P)>0,
\qquad
\sum_{i\in X}p_i(P)=1.
$$

The stationary Markov measure $\mu_P$ on $M$ is determined by
\begin{equation}\label{eq:stationary-markov-cylinder}
\mu_P([0; i_0,\ldots,i_n])
=
p_{i_0}(P)
P_{i_0i_1}\cdots P_{i_{n-1}i_n}.
\end{equation}
Equivalently, under $\mu_P$, the coordinate process
$(X_k)_{k\in\NN}$ is a stationary Markov chain with stationary
distribution $p(P)$ and transition matrix $P$:
$$
\mu_P(X_0=i)=p_i(P),
\qquad
\mu_P(X_{k+1}=j\mid X_k=i)=P_{ij}.
$$

For every $i\in X$, let
$$
M_i=[0;i]
=
\{x\in M:X_0(x)=i\}.
$$
Since $\mu_P(M_i)=p_i(P)>0$, we may define the conditional Markov
measure
$$
\mu_{P,i}
=
\mu_P(\,\cdot\mid X_0=i).
$$
Thus, if $i_0=i$, then
\begin{equation}\label{eq:conditional-path-measure}
\mu_{P,i}([0;i_0,\ldots,i_n])
=
P_{i_0i_1}\cdots P_{i_{n-1}i_n}.
\end{equation}
In other words, under $\mu_{P,i}$, the future coordinate process
$(X_k)_{k\geq 0}$ is the Markov chain with transition matrix $P$
started from the deterministic state $X_0=i$.

For every integrable measurable function $Z:M\to\RR$, we write
$$
\EE_{P,i}[Z]
=
\int_M Z\,\dd\mu_{P,i}.
$$
When the transition matrix $P$ is fixed and no ambiguity is possible,
we abbreviate this notation to
$$
\EE_i[Z]=\EE_{P,i}[Z].
$$
The expectation with respect to the stationary Markov measure is
denoted by
$$
\EE_P[Z]
=
\int_M Z\,\dd\mu_P.
$$
Conditioning on the value of the time-zero coordinate gives
\begin{equation}\label{eq:stationary-expectation-decomposition}
\EE_P[Z]
=
\sum_{i\in X}p_i(P)\EE_{P,i}[Z].
\end{equation}

Now let
$$
A:X\to\GL(2,\RR),
\qquad
i\longmapsto A_i,
$$
be a locally constant linear cocycle. Equivalently, define
$$
A(x)=A_{X_0(x)}=A_{x_0}.
$$
For every $n\geq 1$, its $n$-step product is
\begin{align}
A^{(n)}(x)
&=
A(\sigma^{n-1}x)\cdots A(\sigma x)A(x)
\nonumber\\
&=
A_{X_{n-1}(x)}\cdots A_{X_1(x)}A_{X_0(x)}
\nonumber\\
&=
A_{x_{n-1}}\cdots A_{x_1}A_{x_0}.
\label{eq:markov-cocycle-product}
\end{align}

Since the alphabet $X$ is finite and every $A_i$ is invertible,
the required integrability assumptions are automatic. The
Furstenberg--Kesten theorem~\cite{FK}, applied to the cocycle and its
inverse, gives two constants
$$
\lambda_+(A,P)\geq\lambda_-(A,P)
$$
such that, for $\mu_P$-almost every $x\in M$,
\begin{equation}\label{eq:extremal-lyapunov-exponents}
\lambda_+(A,P)
=
\lim_{n\to\infty}
\frac1n\log\norm{A^{(n)}(x)}
\end{equation}
and
\begin{align}
\lambda_-(A,P)
&=
-\lim_{n\to\infty}
\frac1n\log\norm{\bigl(A^{(n)}(x)\bigr)^{-1}}
\nonumber\\
&=
\lim_{n\to\infty}
\frac1n
\log
\norm{\bigl(A^{(n)}(x)\bigr)^{-1}}^{-1}.
\label{eq:lower-lyapunov-exponent}
\end{align}
Here $\norm{\cdot}$ denotes any fixed operator norm on
$\mathrm{M}_2(\RR)$. The numbers $\lambda_+(A,P)$ and
$\lambda_-(A,P)$ are called the upper and lower, or extremal,
Lyapunov exponents of the Markov cocycle $(A,P)$.

The determinant identity is
\begin{equation}\label{eq:det-identity}
 \lambda_+(A,P)+\lambda_-(A,P)
 =\sum_i p_i\log|\det A_i|.
\end{equation}

\subsection{Norms and total variation}

For a vector $v=(v_i)_{i=1}^N\in\RR^N$, we use
\begin{equation}\label{eq:l1-vector}
 \|v\|_1=\sum_{i=1}^N|v_i|.
\end{equation}
Thus, for probability vectors $p,q$, the quantity $\|p-q\|_1$ is
the sum of the absolute coordinate differences.  For stochastic
matrices we use the maximum row $\ell^1$-distance
\begin{equation}\label{eq:l1-matrix}
 \|P-Q\|_1=\max_i\sum_j|P_{ij}-Q_{ij}|.
\end{equation}
This convention is natural for coupling one transition step from a
fixed current state.

\subsection{Distance}

For two cocycles on the same transition graph, use the local distance
\begin{equation}\label{eq:distance}
 \Delta((A,P),(B,Q))=
 \max_i\bigl(\|A_i-B_i\|+\|A_i^{-1}-B_i^{-1}\|\bigr)
 +\|P-Q\|_1.
\end{equation}
All constants below are uniform in a fixed compact neighborhood on which
the transition graph does not change.

\begin{theoremA}[Main Theorem]
Suppose
$$
 \lambda_+(A,P)=\lambda_-(A,P)=:\lambda_0.
$$
Then there are $C>0$ and $\delta_0\in(0,1)$ such that
\begin{equation}\label{eq:main-logholder}
 |\lambda_\pm(B,Q)-\lambda_0|
 \le C\left(\log\frac1{\Delta((A,P),(B,Q))}\right)^{-1/2}
\end{equation}
whenever $0<\Delta((A,P),(B,Q))<\delta_0$.  In the compact normal
form the power $1/2$ may be replaced by $1$.
\end{theoremA}

The proof follows the zero-gap scheme of Tall--Viana \cite{TV}; the new
points are that the structural reduction is made at Markov return times
and that additive walks are edge observables rather than independent
increments.
\subsection{Previous results}
One of the central questions in the theory of linear cocycles is the
\emph{regularity problem}: how do the Lyapunov exponents depend on the
cocycle and on the driving dynamics? This question has been studied
since Furstenberg's foundational work on random matrix
products~\cite{Furstenberg63}.

Furstenberg--Kifer~\cite{FK} and Hennion~\cite{Hennion84} established
early continuity results under irreducibility-type assumptions.
Bocker--Viana~\cite{BV17} proved continuity for compactly supported
random \(GL(2,\mathbb R)\)-cocycles, and Malheiro--Viana~\cite{MV} developed the corresponding theory for finite state Markov cocycles. Continuity was subsequently extended to cocycles with invariant holonomies by Backes--Brown--Butler~\cite{BBB18}, while Avila--Eskin--Viana~\cite{AEV} proved continuity for compactly supported random products in arbitrary dimension. The compact-support assumption is relevant, as continuity may fail without it~\cite{SV20}.

Quantitative versions of these continuity results have also been established.
Le~Page~\cite{LePage} proved Hölder continuity of the top exponent under strong irreducibility and contraction assumptions. In dimension two, Tall--Viana~\cite{TV} obtained pointwise Hölder continuity at every compactly supported distribution with simple spectrum, and a log-Hölder modulus without the simplicity assumption.
For finitely supported measures with distinct exponents, Duarte and
Klein~\cite{DK20} obtained a weak-Hölder continuity.
More recently, Araújo~\cite{Araujo26} extended the pointwise Hölder
conclusion of Tall--Viana to random products in arbitrary fiber
dimension, assuming that the top exponent is simple and that the
associated equator has dimension at most one. Here the equator is the
maximal invariant subspace on which the asymptotic growth is strictly
smaller than the top exponent. Liu--Viana~\cite{LV26} proved pointwise log-Hölder continuity at compactly supported semisimple \(GL(d,\mathbb R)\)-distributions with one-point Lyapunov spectrum. Their theorem applies in arbitrary dimension, but the semisimplicity assumption excludes the non split triangular case already present in dimension two. In the Markov setting, Cai, Durães, Klein and Melo~\cite{CDKM} proved joint local Hölder continuity of the top Lyapunov exponent with respect to the cocycle and the transition kernel at quasi-irreducible cocycles with simple top exponent.

The present work treats the full zero-gap $GL(2,\mathbb R)$ setting for finite state Markov cocycles.

\section{Outline of the proof}
\label{sec:proof-comments}
This section describes the logic of the arguments before the technical estimates are introduced. 

For two probability measures $\mu,\nu$ on the same measurable space,
we define their total variation \emph{distance} by
\begin{equation}\label{eq:tv-distance-definition}
 d_{\rm TV}(\mu,\nu)=\sup_E|\mu(E)-\nu(E)|.
\end{equation}
On a finite set this becomes
\begin{equation}\label{eq:tv-finite}
 d_{\rm TV}(\mu,\nu)
 =\frac12\sum_x|\mu(x)-\nu(x)|
 =\frac12\|\mu-\nu\|_1.
\end{equation}
These equivalent descriptions, as well as the coupling inequality used
below, are recalled in Roch's coupling chapter
\cite[Lemmas~4.1.9 and~4.1.11]{Roch}.
There is a second convention that is also common.  The total variation
\emph{norm} of the signed measure $\mu-\nu$ is
\begin{equation}\label{eq:variation-norm}
 \|\mu-\nu\|_{\rm var}
 :=\sup_{\|f\|_\infty\le1}
   \left|\int f\dd\mu-\int f\dd\nu\right|
 =2d_{\rm TV}(\mu,\nu).
\end{equation}
Throughout this article, $d_{\rm TV}$ means the probability distance
\eqref{eq:tv-distance-definition}.

\subsection{Couplings and path laws}
\label{subsec:coupling-conventions}

For probability measures $\alpha,\beta$ on a measurable space
$S$, write $\Cpl(\alpha,\beta)$ for the set of probability
measures $\pi$ on $S\times S$ whose coordinate marginals are
$\alpha$ and $\beta$.  If
$$
 U(u,v)=u,\qquad V(u,v)=v,
$$
then $U\sim\alpha$, $V\sim\beta$, and
\begin{equation}\label{eq:coupling-disagreement-event}
 \Prob_\pi(U\ne V)
 =\pi\{(u,v):u\ne v\}.
\end{equation}
The coupling inequality is
\begin{equation}\label{eq:coupling-inequality-intro}
 d_{\rm TV}(\alpha,\beta)
 \le \Prob_\pi(U\ne V).
\end{equation}
A coupling is \emph{maximal} when equality holds.  On a finite space,
one obtains such a coupling by putting mass
$\min\{\alpha(x),\beta(x)\}$ at each diagonal point $(x,x)$,
and then coupling the two residual measures.  These conventions and
constructions are those of
\cite[Definition~4.1.1 and Lemmas~4.1.11, 4.1.13]{Roch}.

For a transition matrix $R$ and initial law $a$, let
\begin{equation}\label{eq:path-law-notation}
 \Law_{a,R}^{[0,n]}
 :=\Law_{a,R}(X_0,\ldots,X_n).
\end{equation}
When $a=p(R)$ is the stationary vector, we abbreviate this to
$\Law_R^{[0,n]}$.  A coupling of path laws is therefore a probability
measure in
$\Cpl(\Law_{a,P}^{[0,n]},\Law_{b,Q}^{[0,n]})$; its probability
symbol always refers to that product probability space.

\subsection{Finite scale upper approximation}
For any cocycle define the finite scale upper approximation
$$
 F_n(A,P)=\frac1n\EE_P\log\|A^{(n)}(X)\|.
$$
Subadditivity gives
$$
 \lambda_+(A,P)=\inf_{n\ge1}F_n(A,P),
 \qquad
 \lambda_+(A,P)\le F_n(A,P).
$$
For a fixed time $n$, matrix products and Markov path probabilities
depend regularly on the parameters.  The cost of this finite time
stability grows exponentially with $n$: schematically,
\begin{equation}\label{eq:comments-finite-time}
 |F_n(B,Q)-F_n(A,P)|\le C\delta e^{bn},
 \qquad
 \delta=\Delta((A,P),(B,Q)).
\end{equation}
The proof of the theorem therefore has two parts:
\begin{enumerate}[label=\textup{(\arabic*)}]
\item obtain a quantitative approximation of the infinite time
      exponent by an object at time $n$;
\item choose $n$ as a function of $\delta$ so that this dynamical
      approximation error balances the perturbative error in
      \eqref{eq:comments-finite-time}.
\end{enumerate}
Since the second error contains $e^{bn}$, the natural choice is
always $n\asymp\log(1/\delta)$.  A polynomial error in $n$ then
becomes a logarithmic modulus in $\delta$.

\subsection{Main idea}

Assume $\lambda_+=\lambda_-=\lambda_0$.  First remove the common
scalar growth, so the normalized cocycle has both exponents equal to
zero.  Inducing at the returns to a fixed Markov state produces i.i.d.
return blocks.  Furstenberg's two-dimensional alternative then reduces
the normalized return cocycle to one of the familiar degenerate forms:
compact, diagonal, triangular, or an interchange of two directions.
The return construction is used only for the algebraic classification.

In diagonal coordinates, logarithmic growth is governed by a centered
additive functional of the Markov chain.  In triangular coordinates,
the diagonal ratio produces the same additive walk and the upper right
entry produces a Markov perpetuity.  These increments are not
independent at deterministic time.  The Poisson equation decomposes a
centered edge observable as
$$
 \text{edge observable}
 =\text{martingale difference}+\text{bounded coboundary}.
$$
The martingale maximal inequality gives diffusive size $O(\sqrt n)$
for the additive walk.  In the degenerate triangular case, the same
maximal estimate controls the largest partial sum entering the
perpetuity, and hence gives
$$
 0\le F_n(A,P)-\lambda_0\le Cn^{-1/2}.
$$
Combining this with \eqref{eq:comments-finite-time} and taking
$n\asymp\log(1/\delta)$ gives
$$
 |\lambda_\pm(B,Q)-\lambda_0|
 \le C\bigl(\log(1/\delta)\bigr)^{-1/2}.
$$
In the compact case the normalized products are uniformly bounded, so
the approximation error is $O(1/n)$, which yields the stronger
$(\log(1/\delta))^{-1}$ modulus.

\section{Markov preliminaries and finite time reduction}
\label{sec:prelim}

\begin{lemma}
\label{lem:path-laws}
In a neighborhood of a primitive $P$, the stationary vector
$p(Q)$ depends Lipschitz continuously on $Q$.  Moreover, the
stationary laws of paths $(X_0,\ldots,X_n)$ for $P$ and $Q$ may
be coupled so that
\begin{equation}\label{eq:path-tv}
 d_{\rm TV}\bigl(\Law_P^{[0,n]}(X_0,\ldots,X_n),
                   \Law_Q^{[0,n]}(X_0,\ldots,X_n)\bigr)
 \le C(n+1)\|P-Q\|_1.
\end{equation}
\end{lemma}

\begin{proof}
Write row vectors on the left and let
$$
 H_0=\left\{v\in\RR^N:\sum_iv_i=0\right\}.
$$
Since $P$ is irreducible, the eigenvalue $1$ is simple.  Therefore
the map
$$
 v\longmapsto v(I-P)
$$
is invertible from \(H_0\) onto itself.  If $q=q(Q)$, then
$q(I-Q)=0$, while $p(I-P)=0$.  Subtracting gives
\begin{equation}\label{eq:stationary-vector-linearization}
 (q-p)(I-P)=q(Q-P).
\end{equation}
Both sides belong to $H_0$.  Applying the inverse of $I-P$ on
$H_0$, and using $\|q\|_1=1$, yields
\begin{equation}\label{eq:stationary-vector-lipschitz-expanded}
 \|q-p\|_1\le C_P\|Q-P\|_1.
\end{equation}
The same inverse remains uniformly bounded in a sufficiently small
neighborhood of $P$, proving the first assertion.

Now define
$$
 E_k=\{X_0=Y_0,\ldots,X_k=Y_k\}.
$$
Thus, $E_k$ is the event that the two paths coincide through time
$k$.

We first maximally couple the initial stationary distributions
$$
 X_0\sim p(P),
 \qquad
 Y_0\sim p(Q).
$$
By the defining property of a maximal coupling,
\begin{equation}\label{eq:initial-maximal-coupling}
 \Prob(X_0\ne Y_0)
 =
 d_{\rm TV}\bigl(p(P),p(Q)\bigr)
 =
 \frac12\|p(P)-p(Q)\|_1.
\end{equation}

Now suppose that $E_k$ occurs and that the common state at time $k$ is $i$.  Conditionally on this event, the laws of the next states are
$$
 \Law_P(X_{k+1}\mid E_k,X_k=Y_k=i)=P_{i\bullet},
$$
and
$$
 \Law_Q(Y_{k+1}\mid E_k,X_k=Y_k=i)=Q_{i\bullet}.
$$
Choose a maximal coupling of these two probability vectors.  Then
\begin{align}
 &\Prob\bigl(
   X_{k+1}\ne Y_{k+1}
   \mid E_k,X_k=Y_k=i
 \bigr)
 \notag\\
 &\qquad
 =
 d_{\rm TV}(P_{i\bullet},Q_{i\bullet})
 =
 \frac12\sum_j|P_{ij}-Q_{ij}|.
 \label{eq:one-step-maximal-coupling}
\end{align}
Since
$$
 \|P-Q\|_1
 =
 \max_i\sum_j|P_{ij}-Q_{ij}|,
$$
equation \eqref{eq:one-step-maximal-coupling} gives, uniformly in
\(i\),
\begin{equation}\label{eq:one-step-disagreement-bound}
 \Prob\bigl(
   X_{k+1}\ne Y_{k+1}
   \mid E_k,X_k=Y_k=i
 \bigr)
 \le \frac12\|P-Q\|_1.
\end{equation}

We now estimate the probability that the two complete paths differ.
Define the first disagreement events
$$
 D_{-1}=\{X_0\ne Y_0\},
$$
and, for $0\le k\le n-1$,
$$
 D_k
 =
 E_k\cap\{X_{k+1}\ne Y_{k+1}\}.
$$
Thus, $D_k$ is the event that the paths agree through time $k$ and
separate at the transition from time $k$ to time $k+1$.  These
events are pairwise disjoint, and
\begin{equation}\label{eq:path-disagreement-decomposition}
 \bigl\{
   (X_0,\ldots,X_n)\ne(Y_0,\ldots,Y_n)
 \bigr\}
 =
 D_{-1}\cup\bigcup_{k=0}^{n-1}D_k.
\end{equation}

Using conditional probability and
\eqref{eq:one-step-disagreement-bound}, we obtain
\begin{align}
 \Prob(D_k)
 &=
 \sum_i
 \Prob(E_k,X_k=Y_k=i)
 \notag\\
 &\qquad\quad\times
 \Prob\bigl(
   X_{k+1}\ne Y_{k+1}
   \mid E_k,X_k=Y_k=i
 \bigr)
 \notag\\
 &\le
 \frac12\|P-Q\|_1
 \sum_i\Prob(E_k,X_k=Y_k=i)
 \notag\\
 &=
 \frac12\|P-Q\|_1\Prob(E_k)
 \notag\\
 &\le
 \frac12\|P-Q\|_1.
 \label{eq:first-disagreement-bound}
\end{align}

Therefore, by \eqref{eq:initial-maximal-coupling},
\eqref{eq:path-disagreement-decomposition}, and
\eqref{eq:first-disagreement-bound},
\begin{align}
 &\Prob\bigl(
   (X_0,\ldots,X_n)\ne(Y_0,\ldots,Y_n)
 \bigr)
 \notag\\
 &\qquad
 =
 \Prob(D_{-1})
 +\sum_{k=0}^{n-1}\Prob(D_k)
 \notag\\
 &\qquad
 \le
 \frac12\|p(P)-p(Q)\|_1
 +\frac n2\|P-Q\|_1.
 \label{eq:path-disagreement-expanded}
\end{align}
Using the Lipschitz estimate for the stationary distributions,
$$
 \|p(P)-p(Q)\|_1
 \le C_P\|P-Q\|_1,
$$
we conclude that
\begin{align}
 \Prob\bigl(
   (X_0,\ldots,X_n)\ne(Y_0,\ldots,Y_n)
 \bigr)
 &\le
 \frac{C_P+n}{2}\|P-Q\|_1
 \notag\\
 &\le
 C(n+1)\|P-Q\|_1.
 \label{eq:path-coupling-expanded}
\end{align}

After the two chains separate, they may be coupled in an arbitrary
way. 
%
Finally, since the joint law of
$$
 \bigl((X_0,\ldots,X_n),(Y_0,\ldots,Y_n)\bigr)
$$
is a coupling of the two path laws, the coupling inequality gives
$$
 d_{\rm TV}\bigl(
   \Law_P^{[0,n]}(X_0,\ldots,X_n),
   \Law_Q^{[0,n]}(Y_0,\ldots,Y_n)
 \bigr)
 \le
 C(n+1)\|P-Q\|_1.
$$

\end{proof}

\subsection{Finite time perturbation}
Let
\begin{equation}\label{eq:Fn}
 F_n(A,P)=\frac1n\int_M\log\|A^{(n)}(x)\|\,\dd\mu_P(x).
\end{equation}
Subadditivity gives $\lambda_+(A,P)=\inf_nF_n(A,P)$.

\begin{lemma}
\label{lem:finite-time-stability}
There exist constants $C_0,C,b>0$ such that, if
$$
 \delta=\Delta((A,P),(B,Q)),
$$
then
\begin{equation}\label{eq:finite-time-stability}
 |F_n(B,Q)-F_n(A,P)|\le C\delta e^{bn}
\end{equation}
whenever \(C_0\delta e^{bn}\le 1/2\).
\end{lemma}

\begin{proof}
Choose $K\ge2$ such that, throughout the fixed neighborhood,
$$
 \max_i\{\|A_i\|,\|A_i^{-1}\|,
          \|B_i\|,\|B_i^{-1}\|\}\le K.
$$
For an admissible word
$\omega=(i_0,\ldots,i_{n-1})$, write
$$
 A_\omega^{(n)}=A_{i_{n-1}}\cdots A_{i_0},
 \qquad
 B_\omega^{(n)}=B_{i_{n-1}}\cdots B_{i_0}.
$$
The telescoping identity gives
$$
 B_\omega^{(n)}-A_\omega^{(n)}
 =
 \sum_{r=0}^{n-1}
 B_{i_{n-1}}\cdots B_{i_{r+1}}
 (B_{i_r}-A_{i_r})
 A_{i_{r-1}}\cdots A_{i_0},
$$
and hence
$$
 \|B_\omega^{(n)}-A_\omega^{(n)}\|
 \le nK^{n-1}\delta.
$$
Moreover,
$\|A_\omega^{(n)}\|\ge K^{-n}$, because
$\|(A_\omega^{(n)})^{-1}\|\le K^n$. Therefore,
$$
 \frac{\bigl|\|B_\omega^{(n)}\|
                   -\|A_\omega^{(n)}\|\bigr|}
      {\|A_\omega^{(n)}\|}
 \le nK^{2n-1}\delta
 \le C_0\delta e^{bn},
$$
after choosing $C_0,b>0$ sufficiently large. If
$C_0\delta e^{bn}\le1/2$, the inequality
$|\log(1+t)|\le2|t|$ yields, uniformly in $\omega$,
\begin{equation}\label{eq:word-log-difference-expanded}
 \left|
 \log\|B_\omega^{(n)}\|-\log\|A_\omega^{(n)}\|
 \right|
 \le C\delta e^{bn}.
\end{equation}

Let $\mu_{P,n}$ and $\mu_{Q,n}$ denote the stationary laws of
words of length $n$ associated with $P$ and $Q$, respectively.
They are supported on the same set $\Omega_n$, since $P$ and $Q$
have the same transition graph. Define
$$
 g_{A,n}(\omega)=\frac1n\log\|A_\omega^{(n)}\|,
 \qquad
 g_{B,n}(\omega)=\frac1n\log\|B_\omega^{(n)}\|.
$$
Then
$$
 F_n(A,P)=\int_{\Omega_n}g_{A,n}\,\dd\mu_{P,n},
 \qquad
 F_n(B,Q)=\int_{\Omega_n}g_{B,n}\,\dd\mu_{Q,n}.
$$
Consequently,
\begin{align*}
 |F_n(B,Q)-F_n(A,P)|
 &\le
 \|g_{B,n}-g_{A,n}\|_\infty
 +2\|g_{A,n}\|_\infty
   d_{\rm TV}(\mu_{Q,n},\mu_{P,n}).
\end{align*}
By \eqref{eq:word-log-difference-expanded},
$$
 \|g_{B,n}-g_{A,n}\|_\infty
 \le \frac{C}{n}\delta e^{bn},
$$
while the uniform matrix bounds imply
$\|g_{A,n}\|_\infty\le\log K$. Lemma~\ref{lem:path-laws} gives
$$
 d_{\rm TV}(\mu_{Q,n},\mu_{P,n})
 \le Cn\|P-Q\|_1\le Cn\delta.
$$
It follows that
$$
 |F_n(B,Q)-F_n(A,P)|
 \le \frac{C}{n}\delta e^{bn}+Cn\delta
 \le C\delta e^{bn},
$$
where the last inequality uses $n\le C e^{bn}$. This proves
\eqref{eq:finite-time-stability}.
\end{proof}


Now, let 
\begin{equation}\label{eq:normalization}
 c_i=\frac12\log|\det A_i|,
 \qquad \widehat A_i=e^{-c_i}A_i.
\end{equation}
Then $|\det\widehat A_i|=1$.  If
$\lambda_+(A,P)=\lambda_-(A,P)=\lambda_0$, the determinant identity
gives
\begin{equation}\label{eq:zero-gap-scalar-mean}
 \sum_i p_i c_i
 =\frac12\sum_i p_i\log|\det A_i|
 =\frac12(\lambda_+(A,P)+\lambda_-(A,P))
 =\lambda_0.
\end{equation}
Since the chain is started with law $p$ and $pP^k=p$, the law of
every coordinate $X_k$ is $p$.  Consequently,
\begin{align}
 F_n(A,P)
 &=\frac1n\EE_P\log\left\|
     \exp\left(\sum_{k=0}^{n-1}c_{X_k}\right)
     \widehat A^{(n)}(X)\right\|\notag\\
 &=\lambda_0+
   \frac1n\EE_P\log\|\widehat A^{(n)}(X)\|.
 \label{eq:zero-gap-target-identity}
\end{align}
On the other hand, subadditivity gives
\begin{equation}\label{eq:zero-gap-lower-all-scales}
 F_n(A,P)\ge\inf_{m\ge1}F_m(A,P)
 =\lambda_+(A,P)=\lambda_0
 \qquad\text{for every }n\ge1.
\end{equation}
Thus the lower bound in
$$
 0\le F_n(A,P)-\lambda_0\le Cn^{-\sigma}
$$
is automatic at every scale.  The entire dynamical problem is the
upper bound, and by \eqref{eq:zero-gap-target-identity} it is equivalent
to
\begin{equation}\label{eq:normalized-growth-target}
 \EE_P\log\|\widehat A^{(n)}(X)\|
 \le Cn^{1-\sigma}
 \qquad(n\ge1).
\end{equation}

\begin{proposition}
\label{prop:finite-time-principle}
Assume that for some $\sigma\in(0,1]$,
\begin{equation}\label{eq:finite-time-rate}
 0\le F_n(A,P)-\lambda_0\le Cn^{-\sigma}
 \qquad(n\ge1),
\end{equation}
where $\lambda_+(A,P)=\lambda_-(A,P)=\lambda_0$.  Then
\begin{equation}\label{eq:abstract-logholder}
 |\lambda_\pm(B,Q)-\lambda_0|
 \le C'\left(\log\frac1\delta\right)^{-\sigma}
\end{equation}
for all sufficiently small $\delta=\Delta((A,P),(B,Q))$.
\end{proposition}

\begin{proof}
Let $b$ and $C_0$ be the constants in
Lemma~\ref{lem:finite-time-stability} and fix
$$
 c=\frac1{2b},
 \qquad
 n=\left\lfloor c\log\frac1\delta\right\rfloor.
$$
Choose $\delta_0$ so small that $n\ge1$ and
$n\ge\tfrac c2\log(1/\delta)$ for $0<\delta<\delta_0$.  Since
$bn\le bc\log(1/\delta)=\tfrac12\log(1/\delta)$,
\begin{equation}\label{eq:finite-time-smallness-expanded}
 C_0\delta e^{bn}\le C_0\delta^{1/2}.
\end{equation}
Decreasing $\delta_0$ once more makes the last quantity at most
$1/2$, so Lemma~\ref{lem:finite-time-stability} applies.

Subadditivity and \eqref{eq:finite-time-rate} now give
\begin{align}
 \lambda_+(B,Q)-\lambda_0
 &\le F_n(B,Q)-F_n(A,P)+F_n(A,P)-\lambda_0\notag\\
 &\le C\delta e^{bn}+Cn^{-\sigma}\notag\\
 &\le C\delta^{1/2}
 +C\left(\log\frac1\delta\right)^{-\sigma}.
 \label{eq:finite-time-upper-expanded}
\end{align}
For every $a>0$ and $\sigma>0$,
$$
 \lim_{\delta\downarrow0}
 \delta^a\left(\log\frac1\delta\right)^\sigma=0.
$$
Thus the logarithmic term is larger than the positive power
$\delta^{1/2}$ for sufficiently small $\delta$, and
\eqref{eq:finite-time-upper-expanded} proves the desired upper bound for
$\lambda_+$.
For the lower bound, set
$$
 \mathcal D(B,Q)=\sum_i p_i(Q)\log|\det B_i|.
$$
The stationary-weight estimate in Lemma~\ref{lem:path-laws}, the local
Lipschitz continuity of $B\mapsto\log|\det B|$, and
\eqref{eq:det-identity} give
\begin{equation}\label{eq:determinant-lipschitz-expanded}
 |\mathcal D(B,Q)-2\lambda_0|\le C\delta.
\end{equation}
Since $\lambda_+(B,Q)\ge\lambda_-(B,Q)$,
$$
 \lambda_+(B,Q)
 \ge\frac12\mathcal D(B,Q)
 \ge\lambda_0-C\delta.
$$
Again $\delta$ is dominated by the logarithmic modulus.  Finally,
$$
 \lambda_-(B,Q)=\mathcal D(B,Q)-\lambda_+(B,Q),
$$
and combining \eqref{eq:determinant-lipschitz-expanded} with the two-sided
estimate for $\lambda_+$ gives \eqref{eq:abstract-logholder} for
$\lambda_-$.
\end{proof}

\section{Algebraic classification}
\subsection{First returns and  normal forms}
\label{sec:normal}

Use the determinant normalization fixed in
\eqref{eq:normalization}.  By \eqref{eq:zero-gap-scalar-mean}, both
exponents of \((\widehat A,P)\) vanish.

\begin{lemma}
\label{lem:regeneration}
Fix $i_*\in X$, start the chain at $i_*$, and let
$$
 \tau_0=0,
 \qquad \tau_{r+1}=\inf\{n>\tau_r:X_n=i_*\}.
$$
The excursion words, their lengths, and their matrix products
$$
 G_r=\widehat A_{X_{\tau_{r+1}-1}}\cdots
     \widehat A_{X_{\tau_r}}
$$
are independent and identically distributed.  The return time
$\tau_1$ has an exponential tail.
\end{lemma}

\begin{proof}
Each $\tau_r$ is a stopping time for the filtration
$\mathcal F_n=\sigma(X_0,\ldots,X_n)$.  At time $\tau_r$ the chain
is always at the same deterministic state $i_*$.  The strong Markov
property says that, conditional on $\mathcal F_{\tau_r}$, the future
process $(X_{\tau_r+k})_{k\ge0}$ has the law of a fresh chain started
at $i_*$ and is independent of the past.  Therefore
$$
 \bigl(\tau_{r+1}-\tau_r,
       (X_{\tau_r},\ldots,X_{\tau_{r+1}}),G_r\bigr),
 \qquad r\ge0,
$$
are independent and identically distributed.

For the tail estimate, choose $m\ge1$ so that
$(P^m)_{ii_*}>0$ for every state $i$; primitivity allows the same
$m$ for all $i$.  Put
$$
 a=\min_i(P^m)_{ii_*}>0.
$$
Starting from any state, the probability of visiting $i_*$ at the end
of the next $m$-block is at least $a$.  On the event
$\{\tau_1>rm\}$, apply the Markov property at time $rm$:
\begin{align*}
 \Prob_{i_*}(\tau_1>(r+1)m)
 &\le(1-a)\Prob_{i_*}(\tau_1>rm).
\end{align*}
Induction gives
\begin{equation}\label{eq:return-tail-expanded}
 \Prob_{i_*}(\tau_1>rm)\le(1-a)^r.
\end{equation}
This is the same block minorization mechanism that appears in the
discussion of uniform geometric ergodicity following
\cite[Theorem~4.3.2]{Roch}.  Indeed, our choice of $a$ says precisely
that
$$
 P^m(i,\mathord\cdot)\ge a\,\delta_{i_*}(\mathord\cdot)
 \qquad\text{for every }i.
$$
Thus each $m$-step block has, uniformly over its initial state, a
success probability at least $a$ of reaching $i_*$, so that
$\tau_1/m$ is stochastically dominated by a geometric random variable
with parameter $a$.  Roch applies the analogous minorization to a
coalescence time and then uses the coupling inequality to obtain uniform
geometric convergence to stationarity.  Here the random time is instead
the first return time to $i_*$; the geometric-tail iteration is the
same.
Thus $\tau_1$ has exponential moments in a neighborhood of the
origin, and in particular finite mean.
\end{proof}

\begin{lemma}
\label{lem:return-exponent}
Let $\nu_*$ be the law of $G_0$.  Then
\begin{equation}\label{eq:return-exponent}
 \lambda_+(\nu_*)=\EE_{i_*}(\tau_1)\lambda_+(\widehat A,P)=0.
\end{equation}
\end{lemma}

\begin{proof}
The products concatenate exactly:
\begin{equation}\label{eq:return-products-concatenate}
 G_{r-1}\cdots G_0=\widehat A^{(\tau_r)}(X).
\end{equation}
By Lemma~\ref{lem:regeneration}, the increments
$\tau_{k+1}-\tau_k$ are i.i.d. and integrable.  Hence the strong law
gives
\begin{equation}\label{eq:return-time-slln-expanded}
 \frac{\tau_r}{r}\longrightarrow\EE_{i_*}\tau_1
 \qquad\text{almost surely}.
\end{equation}
On the other hand, the multiplicative ergodic theorem for the original
Markov cocycle and the i.i.d. product theorem for the return blocks give
\begin{align*}
 \lambda_+(\nu_*)
 &=\lim_{r\to\infty}\frac1r
   \log\|G_{r-1}\cdots G_0\|\\
 &=\lim_{r\to\infty}
   \frac{\tau_r}{r}
   \frac1{\tau_r}\log\|\widehat A^{(\tau_r)}(X)\|\\
 &=\EE_{i_*}(\tau_1)\lambda_+(\widehat A,P).
\end{align*}
The last factor is zero by determinant normalization.  The exponential
tail \eqref{eq:return-tail-expanded} and the uniform matrix bounds
guarantee the required logarithmic integrability of $G_0$.
\end{proof}

\begin{lemma}
\label{lem:return-group}
Let $H$ be the closed subgroup of
$\mathrm{SL}^{\pm}(2,\RR)$ generated by $\operatorname{supp}\nu_*$.
Then either $H$ is compact or $H$ preserves a nonempty finite subset of $\PP^1$.
\end{lemma}

\begin{proof}
By Lemma~\ref{lem:return-exponent}, the i.i.d.\ return law
$\nu_*$ satisfies
$$
 \lambda_+(\nu_*)=0.
$$
Moreover, $\nu_*$ has a finite logarithmic moment and is supported
on $\SL^\pm(2,\RR)$.

Furstenberg's two-dimensional positivity alternative \cite[Theorem~6.11]{Viana} states that if $H$ is noncompact and acts
strongly irreducibly on $\RR^2$, then
$$
 \lambda_+(\nu_*)>0.
$$
Its contrapositive therefore implies that either $H$ is compact or
its action is not strongly irreducible.

In dimension two, failure of strong irreducibility is equivalent to
the existence of a nonempty finite $H$-invariant subset of
$\PP^1$.  This proves the lemma.
\end{proof}

\begin{lemma}
\label{lem:transport}
There are matrices $C_i\in\mathrm{GL}(2,\RR)$ such that every allowed edge matrix
\begin{equation}\label{eq:edge-matrix}
 T_{ij}=C_j^{-1}\widehat A_iC_i,
 \qquad P_{ij}>0,
\end{equation}
belongs to the return group $H$.  Consequently the compact or
finite invariant set alternative of Lemma~\ref{lem:return-group} holds
simultaneously for all allowed edges.
\end{lemma}

\begin{proof}
Choose an admissible path from $i_*$ to each $i$, with normalized
product $C_i$, and a path from $i$ back to $i_*$, with product
$R_i$.  If $i\to j$ is allowed, both $R_jC_j$ and
$R_j\widehat A_iC_i$ are products around based loops at $i_*$.
Splitting a based loop at its intermediate visits to $i_*$ shows that
both products belong to $H$.  Hence
$$
 C_j^{-1}\widehat A_iC_i
 =(R_jC_j)^{-1}(R_j\widehat A_iC_i)\in H.
$$
This groupoid identity is the Markov replacement for a single semigroup
in the i.i.d. \cite{TV} setting.
\end{proof}

\begin{proposition}
\label{prop:normal-forms}
After the state dependent change of coordinates in
\eqref{eq:edge-matrix}, exactly one of the following alternatives holds:
\begin{enumerate}[label=\textup{(\roman*)}]
\item all $T_{ij}$ lie in a compact subgroup (conformal case);
\item all $T_{ij}$ are diagonal in a common ordered pair of lines (degenerate diagonal case);
\item all $T_{ij}$ are diagonal or anti-diagonal in a common unordered pair of lines (simply reducible case);
\item all $T_{ij}$ are upper triangular in a common invariant line,
      and the family is not simultaneously diagonalizable (degenerate triangular case).
\end{enumerate}
In case \textup{(iii)}, the cocycle becomes diagonal on the two-sheet
extension that records which member of the unordered pair is current.
\end{proposition}

\begin{proof}
If $H$ is compact, let $m_H$ be Haar probability and average the
Euclidean inner product:
$$
 \langle v,w\rangle_H
 =\int_H\langle hv,hw\rangle\,\dd m_H(h).
$$
This is positive definite and $H$-invariant.  In an orthonormal basis
for this inner product every $T_{ij}\in H$ is orthogonal, proving
case \textup{(i)}.

Assume that $H$ is not compact and choose a nonempty finite
$H$-invariant set $\mathcal F\subset\PP^1$ of minimal cardinality.
We claim that $\#\mathcal F\le2$.  Otherwise the action on
$\mathcal F$ defines a homomorphism from $H$ to the finite
permutation group of $\mathcal F$.  Its kernel fixes at least three
projective points.  A projective transformation of $\PP^1$ fixing
three points is the identity, so every matrix in the kernel is scalar.
Since $|\det h|=1$, the only scalar possibilities are $I$ and
$-I$.  Both kernel and image are finite; hence $H$ is finite and
compact, a contradiction.

If $\mathcal F=\{E,F\}$, every element of $H$ permutes the two
lines.  In a basis $(e_E,e_F)$, every allowed edge matrix is therefore
diagonal or anti diagonal.  If both lines are individually invariant we
are in case \textup{(ii)}.  Otherwise at least one edge interchanges
them, giving case \textup{(iii)}.  Introducing a sheet
$\varepsilon\in\{0,1\}$ that records the current line turns every
edge into a diagonal one: an anti-diagonal edge changes the sheet and a
diagonal edge preserves it.

If $\mathcal F=\{E\}$, all elements fix $E$ and are upper
triangular in a basis whose first vector spans $E$.  If there is a
second common invariant line, the family is simultaneously diagonal and
belongs to case \textup{(ii)}.  If not, it is the triangular
case \textup{(iv)}. 
\end{proof}

\begin{remark}[Thieullen's measurable normal forms]
The four cases agree with the rotation, centered triangular,
flip-diagonal, and nonzero diagonal alternatives in Thieullen's theorem
as presented in \cite[Theorem~6.12]{Viana}.  Thieullen uses an arbitrary
measurable conjugacy on the base.  Proposition~\ref{prop:normal-forms}
is stronger for our purpose because it produces finitely many bounded
state-dependent matrices $C_i$.
\end{remark}
\section{Martingale coboundary decomposition}
\label{sec:additive}

\begin{lemma}[Markov Poisson equation]
\label{lem:poisson}
Let $P$ be an irreducible stochastic matrix with stationary
probability vector $p$, and let $\xi(i,j)$ be a real function on
the allowed edges. Set
$$
 g(i)=\sum_j P_{ij}\xi(i,j)
$$
and assume that $\xi$ is centered:
$$
 \sum_{i,j}p_iP_{ij}\xi(i,j)=0.
$$
Then there exists a unique function $u:X\to\RR$ satisfying
\begin{equation}\label{eq:poisson}
 (I-P)u=g,
 \qquad
 \sum_i p_i u(i)=0.
\end{equation}
Let $\Pi$ be the stationary projection
$$
 (\Pi h)(i)
 =
 \left(\sum_j p_jh(j)\right)\1(i).
$$
Then
\begin{equation}\label{eq:fundamental-matrix}
 u=(I-P+\Pi)^{-1}g.
\end{equation}
If $P$ is primitive, then
\begin{equation}\label{eq:potential-series}
 u(i)
 =
 \sum_{m=0}^{\infty}(P^mg)(i)
 =
 \sum_{m=0}^{\infty}\EE_i[g(X_m)].
\end{equation}
\end{lemma}

\begin{proof}
By the centering assumption,
$$
 \sum_i p_i g(i)
 =
 \sum_{i,j}p_iP_{ij}\xi(i,j)
 =
 0.
$$
Let
$$
 H_0
 =
 \left\{
   h:X\to\RR:
   \sum_i p_i h(i)=0
 \right\}.
$$
Since $pP=p$, the range of $I-P$ is contained in $H_0$.
Irreducibility implies
$$
 \ker(I-P)=\operatorname{span}\{\1\}.
$$
Thus $\operatorname{range}(I-P)$ and $H_0$ both have dimension
$|X|-1$, and hence
$$
 \operatorname{range}(I-P)=H_0.
$$
Since $g\in H_0$, the equation $(I-P)u=g$ has a solution.
Any two solutions differ by a constant, so the normalization
$\sum_i p_i u(i)=0$ makes the solution unique.

We next show that $I-P+\Pi$ is invertible. If
$$
 (I-P+\Pi)h=0,
$$
multiplication by $p$ gives
$$
 \sum_i p_i h(i)=0.
$$
Hence $\Pi h=0$, and therefore $(I-P)h=0$. Thus $h$ is
constant and centered, so $h=0$. Therefore $I-P+\Pi$ is
injective and, since the space is finite-dimensional, invertible.

Set
$$
 u=(I-P+\Pi)^{-1}g.
$$
Since
$$
 (I-P+\Pi)u=g,
$$
multiplication by $p$ gives
$$
 \sum_i p_i u(i)
 =
 \sum_i p_i g(i)
 =
 0.
$$
Thus $\Pi u=0$, and the preceding equation reduces to
$(I-P)u=g$. This proves \eqref{eq:fundamental-matrix}.

Suppose now that $P$ is primitive. Then there exist $C>0$ and
$\rho\in(0,1)$ such that
$$
 \|P^m-\Pi\|\le C\rho^m.
$$
Since $\Pi g=0$,
$$
 \|P^mg\|
 =
 \|(P^m-\Pi)g\|
 \le
 C\rho^m\|g\|.
$$
Therefore the series
$$
 \widetilde{u}
 =
 \sum_{m=0}^{\infty}P^mg
$$
converges absolutely. Its partial sums satisfy
$$
 (I-P)\sum_{m=0}^{N}P^mg
 =
 g-P^{N+1}g.
$$
Letting $N\to\infty$ gives
$$
 (I-P)\widetilde{u}=g.
$$
Moreover, $pP^mg=pg=0$ for every $m$, so $\widetilde{u}$ is
centered. Uniqueness therefore gives
$$
 u=\widetilde{u}
   =\sum_{m=0}^{\infty}P^mg.
$$
Finally,
$$
 (P^mg)(i)=\EE_i[g(X_m)],
$$
which proves \eqref{eq:potential-series}.
\end{proof}

The following decomposition is the finite state Markov version of the
Gordin martingale coboundary method; see Gordin~\cite{Gordin1969},
Gordin--Lifshits~\cite{GordinLifshits1978}, and
Glynn--Meyn~\cite{GlynnMeyn1996}.  General treatments through the
Markov Poisson equation may also be found in
Meyn--Tweedie~\cite[Chapter~17]{MeynTweedie}.
%
%
%
\begin{lemma}[martingale coboundary decomposition]
\label{lem:martingale-coboundary}
Define
\begin{equation}\label{eq:martingale-difference}
 D_{k+1}=\xi(X_k,X_{k+1})-g(X_k)
 +u(X_{k+1})-(Pu)(X_k).
\end{equation}
Then $M_n=\sum_{k=0}^{n-1}D_{k+1}$ is a martingale and
\begin{equation}\label{eq:martingale-coboundary}
 S_n:=\sum_{k=0}^{n-1}\xi(X_k,X_{k+1})
 =M_n+u(X_0)-u(X_n).
\end{equation}
\end{lemma}

\begin{proof}
Let $\mathcal F_k=\sigma(X_0,\ldots,X_k)$.  Since the conditional law
of $X_{k+1}$ given $\mathcal F_k$ is the row
$P_{X_k\bullet}$,
\begin{align*}
 \EE\bigl(\xi(X_k,X_{k+1})\mid\mathcal F_k\bigr)&=g(X_k),\\
 \EE\bigl(u(X_{k+1})\mid\mathcal F_k\bigr)&=(Pu)(X_k).
\end{align*}
Therefore $\EE(D_{k+1}\mid\mathcal F_k)=0$, so $(M_n)$ is a
martingale.  Because the Poisson equation says $g=u-Pu$,
\begin{align*}
 D_{k+1}
 &=\xi(X_k,X_{k+1})-u(X_k)+u(X_{k+1}),
\end{align*}
or equivalently
$$
 \xi(X_k,X_{k+1})
 =D_{k+1}+u(X_k)-u(X_{k+1}).
$$
Summing from $k=0$ to $n-1$ telescopes the last two terms and gives
\eqref{eq:martingale-coboundary} exactly.
\end{proof}

\begin{lemma}
\label{lem:diffusion}
There are constants \(C,c>0\) such that
\begin{equation}\label{eq:diffusion-tail}
 \Prob\left(\max_{m\le n}|S_m|>C+t\right)
 \le2\exp\left(-c\frac{t^2}{n}\right),
\end{equation}
and
\begin{equation}\label{eq:diffusion-mean}
 \EE|S_n|+\EE\max_{m\le n}|S_m|\le C\sqrt n.
\end{equation}
\end{lemma}

\begin{proof}
Since $X$ is finite and $\xi$ is defined on the finite set
of allowed edges, there exists $L<\infty$ such that
\begin{equation}\label{eq:bounded-martingale-increments}
 |D_{k+1}|\le L
 \qquad\text{almost surely for every }k.
\end{equation}
The bound is independent of the initial state of the chain.

If $L=0$, then $M_m=0$ for every $m$, and the conclusion follows
immediately from the bounded coboundary term.  We therefore assume
$L>0$.

We know that $\EE[D_{k+1}\mid\mathcal F_k]=0$
and $-L\le D_{k+1}\le L$.
The conditional version of Hoeffding's lemma therefore gives, for every $h\in\RR$,
\begin{equation}\label{eq:conditional-hoeffding}
 \EE\left[
   e^{hD_{k+1}}
   \mid\mathcal F_k
 \right]
 \le
 \exp\left(\frac{h^2L^2}{2}\right).
\end{equation}
For a fixed $h\in\RR$, define
\begin{equation}\label{eq:exponential-supermartingale}
 Z_m(h)
 =
 \exp\left(
   hM_m-\frac12h^2L^2m
 \right).
\end{equation}
Then
\begin{align}
 \EE[Z_{m+1}(h)\mid\mathcal F_m]
 &=
 \EE\left[
   \exp\left(
     hM_m+hD_{m+1}
     -\frac12h^2L^2(m+1)
   \right)
   \,\middle|\,\mathcal F_m
 \right]
 \notag\\
 &=
 Z_m(h)
 \EE\left[
   \exp\left(
     hD_{m+1}-\frac12h^2L^2
   \right)
   \,\middle|\,\mathcal F_m
 \right]
 \notag\\
 &\le Z_m(h),
\end{align}
where the last inequality follows from
\eqref{eq:conditional-hoeffding}.  Hence
$$
 (Z_m(h),\mathcal F_m)_{m\ge0}
$$
is a nonnegative supermartingale, with $Z_0(h)=1$.

Fix $t>0$ and $h>0$, and define the stopping time $\tau = \inf\{0\le m\le n:M_m\ge t\}$,
with the convention $\tau=n+1$ if the set is empty.  On the event
 $\left\{\max_{0\le m\le n}M_m\ge t\right\}$, we have $\tau\le n$, and therefore
\begin{align}
 Z_\tau(h)
 &=
 \exp\left(
   hM_\tau-\frac12h^2L^2\tau
 \right)
 \notag\\
 &\ge
 \exp\left(
   ht-\frac12h^2L^2n
 \right).
 \label{eq:stopped-exponential-lower-bound}
\end{align}

The optional-stopping inequality for the nonnegative supermartingale
$Z_m(h)$ gives
$$
 \EE[Z_{\tau\wedge n}(h)]\le Z_0(h)=1 \qquad \text{ where } \tau \wedge n = \min\{\tau, n\}
$$
Using \eqref{eq:stopped-exponential-lower-bound}, we obtain
\begin{align}
 1
 &\ge
 \EE[Z_{\tau\wedge n}(h)]
 \notag\\
 &\ge
 \exp\left(
   ht-\frac12h^2L^2n
 \right)
 \Prob\left(
   \max_{0\le m\le n}M_m\ge t
 \right).
\end{align}
Consequently,
\begin{equation}\label{eq:martingale-positive-tail-h}
 \Prob\left(
   \max_{0\le m\le n}M_m\ge t
 \right)
 \le
 \exp\left(
   -ht+\frac12h^2L^2n
 \right).
\end{equation}

The exponent on the right is minimized by $h=t/nL^2$.
Substituting this value into
\eqref{eq:martingale-positive-tail-h} gives
\begin{equation}\label{eq:martingale-positive-tail}
 \Prob\left(
   \max_{0\le m\le n}M_m\ge t
 \right)
 \le
 \exp\left(
   -\frac{t^2}{2nL^2}
 \right).
\end{equation}

Apply the same argument to the martingale $-M_m$ gives 
\begin{equation}\label{eq:martingale-maximal-expanded}
 \Prob\left(
   \max_{0\le m\le n}|M_m|\ge t
 \right)
 \le
 2\exp\left(
   -\frac{t^2}{2nL^2}
 \right).
\end{equation}

Since
$$
 |S_m|
 \le
 |M_m|+|u(X_0)|+|u(X_m)|
 \le
 |M_m|+2\|u\|_\infty.
$$
Therefore,
\begin{equation}\label{eq:max-S-by-max-M}
 \max_{0\le m\le n}|S_m|
 \le
 \max_{0\le m\le n}|M_m|
 +2\|u\|_\infty.
\end{equation}
It follows that
$$
 \left\{
   \max_{0\le m\le n}|S_m|
   >
   2\|u\|_\infty+t
 \right\}
 \subseteq
 \left\{
   \max_{0\le m\le n}|M_m|>t
 \right\}.
$$
Using \eqref{eq:martingale-maximal-expanded}, we obtain
\begin{equation}\label{eq:S-maximal-tail-explicit}
 \Prob\left(
   \max_{0\le m\le n}|S_m|
   >
   2\|u\|_\infty+t
 \right)
 \le
 2\exp\left(
   -\frac{t^2}{2nL^2}
 \right).
\end{equation}
This proves \eqref{eq:diffusion-tail}, with
$$
 C=2\|u\|_\infty,
 \qquad
 c=\frac1{2L^2},
$$
after enlarging $C$, if necessary.

For every nonnegative random variable $Y$,
\begin{equation}\label{eq:tail-integral-identity}
 \EE Y
 =
 \int_0^\infty\Prob(Y>s)\,\dd s.
\end{equation}
Apply this to $Y=\max_{0\le m\le n}|S_m|$.
By \eqref{eq:max-S-by-max-M},
\begin{align}
 \EE\max_{0\le m\le n}|S_m|
 &\le
 2\|u\|_\infty
 +
 \EE\max_{0\le m\le n}|M_m|
 \notag\\
 &=
 2\|u\|_\infty
 +
 \int_0^\infty
 \Prob\left(
   \max_{0\le m\le n}|M_m|>t
 \right)
 \,\dd t.
\end{align}
Using \eqref{eq:martingale-maximal-expanded},
\begin{align}
 \EE\max_{0\le m\le n}|S_m|
 &\le
 2\|u\|_\infty
 +
 2\int_0^\infty
 \exp\left(
   -\frac{t^2}{2nL^2}
 \right)
 \,\dd t.
 \label{eq:S-maximal-expectation-integral}
\end{align}
Make the change of variables $t=L\sqrt{2n}\,s$.
Then
\begin{align}
 2\int_0^\infty
 \exp\left(
   -\frac{t^2}{2nL^2}
 \right)
 \,\dd t
 &=
 2L\sqrt{2n}
 \int_0^\infty e^{-s^2}\,\dd s
 \notag\\
 &=
 2L\sqrt{2n}\,\frac{\sqrt\pi}{2}
 \notag\\
 &=
 L\sqrt{2\pi n}.
\end{align}
Thus
$$
 \EE\max_{0\le m\le n}|S_m|
 \le
 2\|u\|_\infty+L\sqrt{2\pi n}.
$$
Since $n\ge1$, the bounded first term can be absorbed into a multiple of $\sqrt n$.  Therefore,
\begin{equation}\label{eq:S-maximal-mean}
 \EE\max_{0\le m\le n}|S_m|
 \le C\sqrt n.
\end{equation}

Finally,
$$
 |S_n|
 \le
 \max_{0\le m\le n}|S_m|,
$$
and hence
$$
 \EE|S_n|
 \le
 \EE\max_{0\le m\le n}|S_m|
 \le C\sqrt n.
$$
After enlarging $C$ once more, we conclude that
$$
 \EE|S_n|
 +
 \EE\max_{0\le m\le n}|S_m|
 \le C\sqrt n.
$$
This proves \eqref{eq:diffusion-mean}.
\end{proof}

\section{Random walk estimates}
\label{sec:zero-gap}
For an admissible path $X_0,\ldots,X_n$, let
\begin{equation}\label{eq:edge-product-Gn}
 G_n=T_{X_{n-1}X_n}\cdots T_{X_0X_1}.
\end{equation}
Products in the edge coordinates satisfy
\begin{equation}\label{eq:endpoint-conjugacy}
 G_n
 =C_{X_n}^{-1}\widehat A^{(n)}(X)C_{X_0}.
\end{equation}
Choose
\begin{equation}\label{eq:LC-definition}
 L_C=\max_i\bigl\{
 \|C_i\|,\|C_i^{-1}\|,
 |\det C_i|,|\det C_i|^{-1}
 \bigr\}\ge1.
\end{equation}
Submultiplicativity gives
$$
 \|G_n\|
 \le L_C^2\|\widehat A^{(n)}(X)\|.
$$
Conversely, solving \eqref{eq:endpoint-conjugacy} for the original
product gives
$$
 \widehat A^{(n)}(X)=C_{X_n}G_nC_{X_0}^{-1},
 \qquad
 \|\widehat A^{(n)}(X)\|\le L_C^2\|G_n\|.
$$
Therefore
\begin{equation}\label{eq:endpoint-norm-bound-detailed}
 \left|\log\|G_n\|-
 \log\|\widehat A^{(n)}(X)\|\right|
 \le2\log L_C.
\end{equation}
Moreover, $|\det\widehat A_i|=1$, so taking determinants in
\eqref{eq:endpoint-conjugacy} gives
$$
 \det G_n
 =\det(C_{X_n}^{-1})\det(C_{X_0})
$$
and hence
\begin{equation}\label{eq:endpoint-determinant-bound-detailed}
 |\log|\det G_n||\le2\log L_C.
\end{equation}

\begin{lemma}[compact, degenerate diagonal and simply reducible estimates]
\label{lem:compact-diagonal}
In case \textup{(i)} of Proposition~\ref{prop:normal-forms},
\begin{equation}\label{eq:compact-growth}
 \EE\log\|\widehat A^{(n)}\|\le C.
\end{equation}
In cases \textup{(ii)} and \textup{(iii)},
\begin{equation}\label{eq:diagonal-growth}
 \EE\log\|\widehat A^{(n)}\|\le C\sqrt n.
\end{equation}
\end{lemma}

\begin{proof}
We treat the three cases separately.

\medskip
\noindent
\textbf{Case (i): compact (conformal).}

This is exactly the Markov analogue of the conformal case in
Tall--Viana~\cite{TV}.  The only apparent difference is that, before the state dependent gauge change, the conformal inner product is allowed to depend on the current Markov state.

Let $\|\cdot\|_H$ be the invariant norm obtained by averaging an
inner product over the compact return group.  Every allowed edge
matrix is an isometry for this norm:
$$
 \|T_{ij}v\|_H=\|v\|_H
 \qquad(P_{ij}>0).
$$
Pull this single edge coordinate norm back to the original fiber over
each state by defining
\begin{equation}\label{eq:state-dependent-conformal-metric}
 \langle v,w\rangle_i
 :=\langle C_i^{-1}v,C_i^{-1}w\rangle_H,
 \qquad v,w\in\RR^2.
\end{equation}
Using $T_{ij}=C_j^{-1}\widehat A_iC_i$, we obtain, for every allowed
transition $i\to j$,
\begin{align*}
 \langle\widehat A_iv,\widehat A_iw\rangle_j
 &=\langle C_j^{-1}\widehat A_iv,
           C_j^{-1}\widehat A_iw\rangle_H\\
 &=\langle T_{ij}C_i^{-1}v,T_{ij}C_i^{-1}w\rangle_H\\
 &=\langle C_i^{-1}v,C_i^{-1}w\rangle_H
  =\langle v,w\rangle_i.
\end{align*}
Thus the determinant normalized cocycle is isometric from the metric
over state $i$ to the metric over state $j$.  Restoring the scalar
factor $A_i=e^{c_i}\widehat A_i$ from
\eqref{eq:normalization} gives the conformal identity
\begin{equation}\label{eq:markov-conformal-identity}
 \langle A_iv,A_iw\rangle_j
 =e^{2c_i}\langle v,w\rangle_i,
 \qquad P_{ij}>0.
\end{equation}

Consequently, every admissible product satisfies
$$
 \|G_nv\|_H=\|v\|_H.
$$
Because all norms on $\RR^2$ are equivalent, there is $K_H\ge1$
such that
$$
 K_H^{-1}\|v\|\le\|v\|_H\le K_H\|v\|.
$$
For $\|v\|=1$, it follows that
$$
 \|G_nv\|
 \le K_H\|G_nv\|_H
 =K_H\|v\|_H
 \le K_H^2.
$$
Thus $\|G_n\|\le K_H^2$ pathwise and uniformly in $n$.  By
\eqref{eq:endpoint-norm-bound-detailed},
$$
 \log\|\widehat A^{(n)}(X)\|
 \le2\log K_H+2\log L_C.
$$
Taking expectations proves \eqref{eq:compact-growth}.

\medskip
\noindent
\textbf{Case (ii): degerate diagonal case.}

Choose vectors spanning the two invariant lines.  In this basis every
allowed edge matrix is diagonal:
$$
 T_{ij}=\begin{pmatrix}a_{ij}&0\\0&d_{ij}\end{pmatrix},
 \qquad a_{ij}d_{ij}\ne0.
$$
The Lyapunov exponents along the two coordinate lines are
$$
 \chi_a=\sum_{i,j}p_iP_{ij}\log|a_{ij}|,
 \qquad
 \chi_d=\sum_{i,j}p_iP_{ij}\log|d_{ij}|.
$$
They are precisely the two exponents of the normalized cocycle.  Since
both normalized exponents vanish,
$$
 \chi_a=\chi_d=0.
$$
Hence the logarithmic ratio observable
$$
 \xi(i,j)=\log\left|\frac{a_{ij}}{d_{ij}}\right|
$$
is centered:
\begin{equation}\label{eq:diagonal-ratio-centered-detailed}
 \sum_{i,j}p_iP_{ij}\xi(i,j)=\chi_a-\chi_d=0.
\end{equation}

Along an admissible path, put
$$
 A_n=\prod_{k=0}^{n-1}a_{X_kX_{k+1}},
 \qquad
 D_n=\prod_{k=0}^{n-1}d_{X_kX_{k+1}}.
$$
Then
$$
 G_n=\begin{pmatrix}A_n&0\\0&D_n\end{pmatrix}
$$
and
\begin{align*}
 S_n
 &=\sum_{k=0}^{n-1}\xi(X_k,X_{k+1})
  =\log|A_n| -\log|D_n|,\\
 R_n
 &=\log|\det G_n|
  =\log|A_n|+\log|D_n|.
\end{align*}
Solving for the two diagonal logarithms gives
$$
 \log|A_n|=\frac12(R_n+S_n),
 \qquad
 \log|D_n|=\frac12(R_n-S_n).
$$
The Euclidean operator norm of a diagonal matrix is the maximum of the
absolute values of its diagonal entries.  Therefore
\begin{align}
 \log\|G_n\|
 &=\max\{\log|A_n|,\log|D_n|\}\notag\\
 &=\frac12R_n+\frac12|S_n|.
 \label{eq:diagonal-norm-exact}
\end{align}
By \eqref{eq:endpoint-determinant-bound-detailed},
$|R_n|\le2\log L_C$, so
\begin{equation}\label{eq:diagonal-norm-pathwise-detailed}
 \log\|G_n\|\le\log L_C+\frac12|S_n|.
\end{equation}
The centering in \eqref{eq:diagonal-ratio-centered-detailed} permits
the use of Lemma~\ref{lem:diffusion}, which gives
$$
 \EE_P|S_n|\le C\sqrt n.
$$
Taking expectations in
\eqref{eq:diagonal-norm-pathwise-detailed} and using
\eqref{eq:endpoint-norm-bound-detailed}, we conclude that
$$
 \EE_P\log\|\widehat A^{(n)}(X)\|
 \le C+\frac12\EE_P|S_n|
 \le C\sqrt n.
$$

\medskip
\noindent
\textbf{Case (iii): Simply reducible case.}

Let $L^0,L^1$ be the invariant unordered pair.  For each allowed
edge, let $s_{ij}\in\{0,1\}$ indicate whether $T_{ij}$ preserves
or interchanges the two lines.  More precisely,
$$
 s_{ij}=0
 \quad\Longleftrightarrow\quad
 T_{ij}L^\varepsilon=L^\varepsilon
 \quad(\varepsilon=0,1),
$$
whereas
$$
 s_{ij}=1
 \quad\Longleftrightarrow\quad
 T_{ij}L^\varepsilon=L^{\varepsilon\oplus1}
 \quad(\varepsilon=0,1).
$$
Here and below $\oplus$ denotes addition in
$\mathbb Z/2\mathbb Z$.

Introduce the two-sheet state space
$$
 X
 =X\times\mathbb Z/2\mathbb Z.
$$
The state $(i,\varepsilon)$ records both the current Markov state
$i$ and the line $L^\varepsilon$ that is currently being followed.
For an allowed base edge $i\to j$, define the lifted edge map by
\begin{equation}\label{eq:two-sheet-transition-detailed}
 (i,\varepsilon)\longrightarrow
 (j,\varepsilon\mathbin\oplus s_{ij}).
\end{equation}
Equivalently, the lifted transition kernel is
\begin{equation}\label{eq:two-sheet-kernel}
 \widetilde P_{(i,\varepsilon),(j,\varepsilon')}
 =P_{ij}\,
  \1_{\{\varepsilon'=\varepsilon\oplus s_{ij}\}}.
\end{equation}
Its rows sum to one because, for each $j$, exactly one value of
$\varepsilon'$ is permitted.  If
$\widetilde X_n=(X_n,\varepsilon_n)$, then the base coordinate
evolves with transition matrix $P$, while the sheet coordinate is
updated deterministically according to
\begin{equation}\label{eq:two-sheet-recursion}
 \varepsilon_{n+1}
 =\varepsilon_n\oplus s_{X_nX_{n+1}},
 \qquad
 \varepsilon_n
 =\varepsilon_0\oplus
   \bigoplus_{k=0}^{n-1}s_{X_kX_{k+1}}.
\end{equation}
Thus the lift introduces no new randomness.

\medskip
\noindent
\emph{Diagonalization on the lift.}
Choose nonzero vectors $v^\varepsilon\in L^\varepsilon$.  There are
nonzero scalars $a_{ij}^{(\varepsilon)}$ such that
\begin{equation}\label{eq:two-sheet-line-multiplier}
 T_{ij}v^\varepsilon
 =a_{ij}^{(\varepsilon)}
  v^{\varepsilon\oplus s_{ij}}.
\end{equation}
At the lifted state $(i,\varepsilon)$, order the two lines as
$(L^\varepsilon,L^{\varepsilon\oplus1})$.  If
$\varepsilon'=\varepsilon\oplus s_{ij}$, then the matrix of
$T_{ij}$ from the ordered basis
$(v^\varepsilon,v^{\varepsilon\oplus1})$ to the ordered basis
$(v^{\varepsilon'},v^{\varepsilon'\oplus1})$ is
\begin{equation}\label{eq:two-sheet-diagonal-edge}
 \widetilde T_{(i,\varepsilon),(j,\varepsilon')}
 =\begin{pmatrix}
   a_{ij}^{(\varepsilon)}&0\\
   0&a_{ij}^{(\varepsilon\oplus1)}
  \end{pmatrix}.
\end{equation}
The underlying linear product is unchanged; only its source and target bases have been reordered.  Since there are only two such bases, passing between the lifted diagonal product and $G_n$ costs a uniform endpoint constant.

\medskip
\noindent
\emph{Stationary data.}
The lifted kernel always has the stationary probability vector
\begin{equation}\label{eq:two-sheet-symmetric-stationary}
 \widetilde p_{i,\varepsilon}=\frac{p_i}{2}.
\end{equation}
Indeed, for every $(j,\varepsilon')$,
\begin{align*}
 \sum_{i,\varepsilon}
 \widetilde p_{i,\varepsilon}
 \widetilde P_{(i,\varepsilon),(j,\varepsilon')}
 &=\sum_{i,\varepsilon}
   \frac{p_i}{2}P_{ij}
   \1_{\{\varepsilon'=\varepsilon\oplus s_{ij}\}}\\
 &=\frac12\sum_i p_iP_{ij}
  =\frac{p_j}{2}.
\end{align*}
The projection $\pi(i,\varepsilon)=i$ therefore sends the stationary
lifted chain to the original stationary chain with data $(p,P)$.

Although $P$ is primitive, $\widetilde P$ need not be irreducible.
There are two possibilities.  If some admissible loop has odd total
swap parity, then the two sheets communicate and the lift is
irreducible; its stationary vector is \eqref{eq:two-sheet-symmetric-stationary}.
If every admissible loop has even swap parity, there is a function
$\varphi:\mathcal I\to\mathbb Z/2\mathbb Z$ satisfying
\begin{equation}\label{eq:two-sheet-coboundary}
 s_{ij}=\varphi(i)\oplus\varphi(j)
 \qquad(P_{ij}>0).
\end{equation}
Indeed, fix a base state $i_*$ and define $\varphi(i)$ as the
swap parity of an admissible path from $i_*$ to $i$.  If two such
paths are chosen, append the same admissible path from $i$ back to
$i_*$.  The resulting based loops have even parity, so the two
definitions agree.  Appending one allowed edge $i\to j$ then gives
$\varphi(j)=\varphi(i)\oplus s_{ij}$, which is
\eqref{eq:two-sheet-coboundary}.  Along every lifted transition, the
quantity $\varepsilon\oplus\varphi(i)$ is therefore constant.
In this case the lift splits into the two closed irreducible classes
$$
 \mathcal C_c
 =\{(i,\varphi(i)\oplus c):i\in\mathcal I\},
 \qquad c\in\mathbb Z/2\mathbb Z,
$$
with classwise stationary probabilities
\begin{equation}\label{eq:two-sheet-class-stationary}
 \widetilde p^{\,c}_{i,\varepsilon}
 =p_i\,
  \1_{\{\varepsilon=\varphi(i)\oplus c\}}.
\end{equation}
The symmetric stationary law is their average:
$\widetilde p=(\widetilde p^{\,0}+\widetilde p^{\,1})/2$.
In either alternative, every recurrent class projects to the stationary
base law, because the projection of any classwise stationary vector is
a stationary vector for $P$, and $p$ is unique.

\medskip
\noindent
\emph{Classwise centering and projection downstairs.}
For a lifted edge
$z=(i,\varepsilon)\to z'=(j,\varepsilon')$, define
\begin{equation}\label{eq:two-sheet-ratio-observable}
 \widetilde\xi(z,z')
 =\log|a_{ij}^{(\varepsilon)}|
  -\log|a_{ij}^{(\varepsilon\oplus1)}|.
\end{equation}
Fix a recurrent class $\mathcal C$, and denote its restricted kernel
and stationary vector by
$\widetilde P^{\mathcal C}$ and
$\widetilde p^{\mathcal C}$.  For
$z=(i,\varepsilon)\to z'=(j,\varepsilon')$, put
$$
 a(z,z')=a_{ij}^{(\varepsilon)},
 \qquad
 d(z,z')=a_{ij}^{(\varepsilon\oplus1)}.
$$
The two diagonal line exponents are
\begin{align*}
 \widetilde\chi_0^{\mathcal C}
 &=\sum_{z,z'\in\mathcal C}
   \widetilde p_z^{\mathcal C}
   \widetilde P_{zz'}^{\mathcal C}
   \log|a(z,z')|,\\
 \widetilde\chi_1^{\mathcal C}
 &=\sum_{z,z'\in\mathcal C}
   \widetilde p_z^{\mathcal C}
   \widetilde P_{zz'}^{\mathcal C}
   \log|d(z,z')|.
\end{align*}
The class projects to the stationary base chain and the endpoint change
of basis is uniformly bounded.  Hence these are the two Lyapunov
exponents of the normalized original cocycle, possibly in the opposite
order.  Both are zero, and therefore
\begin{equation}\label{eq:two-sheet-centered}
 \sum_{z,z'\in\mathcal C}
 \widetilde p_z^{\mathcal C}
 \widetilde P_{zz'}^{\mathcal C}
 \widetilde\xi(z,z')
 =\widetilde\chi_0^{\mathcal C}
  -\widetilde\chi_1^{\mathcal C}=0.
\end{equation}
Lemma~\eqref{lem:diffusion} applied to the finite irreducible chain on $\mathcal C$, now gives the same $O(\sqrt n)$ bound as in the diagonal case.  Periodicity of a class causes no problem: the Poisson equation and martingale coboundary decomposition require only finite irreducibility.

Applying the preceding diagonal calculation on each recurrent class
gives
$$
 \EE_{\widetilde p^{\mathcal C}}
 \log\|G_n\|\le C_{\mathcal C}\sqrt n.
$$
There are only finitely many classes, so their constants may be
replaced by their maximum.  Under the symmetric law
\eqref{eq:two-sheet-symmetric-stationary}, every base path has exactly
two lifts, one for each initial sheet, each with half the probability
of the base path.  Since $G_n$ depends only on the base path,
\begin{equation}\label{eq:two-sheet-projection-expectation}
 \EE_{\widetilde p}\log\|G_n\|
 =\EE_p\log\|G_n\|.
\end{equation}
Thus averaging over the recurrent classes and projecting to the
original chain yields the same estimate downstairs.
Finally, \eqref{eq:endpoint-norm-bound-detailed} transfers it to
$\widehat A^{(n)}$.  This proves \eqref{eq:diagonal-growth} in both
cases (ii) and (iii).

\end{proof}

\subsection{The degenerate triangular case}
Write
\begin{equation}\label{eq:triangular-edge}
 T_{ij}=\begin{pmatrix}a_{ij}&b_{ij}\\0&d_{ij}\end{pmatrix},
 \qquad
 \rho_{ij}=\frac{a_{ij}}{d_{ij}},
 \qquad
 \beta_{ij}=\frac{b_{ij}}{d_{ij}}.
\end{equation}
The first coordinate line is invariant, and the action on the quotient
by that line is multiplication by $d_{ij}$.  Thus 
\begin{equation}\label{eq:triangular-line-exponents-detailed}
 \chi_a=\sum_{i,j}p_iP_{ij}\log|a_{ij}|,
 \qquad
 \chi_d=\sum_{i,j}p_iP_{ij}\log|d_{ij}|
\end{equation}
form the Lyapunov exponents of the triangular cocycle.  Both normalized exponents are zero, and therefore
$$
 \chi_a=\chi_d=0.
$$
Consequently, the logarithmic diagonal ratio is centered:
\begin{equation}\label{eq:critical-drift-article}
 \sum_{i,j}p_iP_{ij}\log|\rho_{ij}|=0.
\end{equation}

\begin{lemma}
\label{lem:affine-coordinate}
If
$$
 G_n=T_{X_{n-1}X_n}\cdots T_{X_0X_1}
 =\begin{pmatrix}A_n&B_n\\0&D_n\end{pmatrix},
 \qquad Z_n=\frac{B_n}{D_n},
$$
then
\begin{equation}\label{eq:affine-recursion}
 Z_{n+1}=\rho_nZ_n+\beta_n,
\end{equation}
and
\begin{equation}\label{eq:finite-perpetuity}
 Z_n=\sum_{k=0}^{n-1}\beta_k
       \prod_{\ell=k+1}^{n-1}\rho_\ell.
\end{equation}
\end{lemma}

\begin{proof}
Abbreviate
$a_n=a_{X_nX_{n+1}}$, and similarly for $b_n,d_n,\rho_n,\beta_n$.
Multiplication on the left gives
$$
 \begin{pmatrix}a_n&b_n\\0&d_n\end{pmatrix}
 \begin{pmatrix}A_n&B_n\\0&D_n\end{pmatrix}
 =\begin{pmatrix}
   a_nA_n&a_nB_n+b_nD_n\\0&d_nD_n
  \end{pmatrix}.
$$
Thus
$$
 B_{n+1}=a_nB_n+b_nD_n,
 \qquad D_{n+1}=d_nD_n.
$$
Since every triangular edge matrix is invertible, $d_n\ne0$.
Dividing by $D_{n+1}=d_nD_n$ gives
$$
 Z_{n+1}=\frac{B_{n+1}}{D_{n+1}}
 =\frac{a_n}{d_n}Z_n+\frac{b_n}{d_n}
 =\rho_nZ_n+\beta_n.
$$
Starting from $Z_0=0$, the first iterates are
$$
 Z_1=\beta_0,
 \quad
 Z_2=\rho_1\beta_0+\beta_1,
 \quad
 Z_3=\rho_2\rho_1\beta_0+\rho_2\beta_1+\beta_2.
$$
The general formula \eqref{eq:finite-perpetuity} follows by induction;
the empty product corresponding to $k=n-1$ equals $1$.
\end{proof}

\begin{lemma}
\label{lem:critical-perpetuity}
For the recursion $Z_n$ in \eqref{eq:affine-recursion}, there are $C,c>0$ such
that
\begin{equation}\label{eq:perpetuity-tail-article}
 \Prob\left(
  \log(1+|Z_n|)>C+\log(n+1)+t
 \right)
 \le C\exp\left(-c\frac{t^2}{n}\right),
\end{equation}
and
\begin{equation}\label{eq:perpetuity-mean-article}
 \EE\log(1+|Z_n|)\le C\sqrt n+C\log(n+1).
\end{equation}
\end{lemma}

\begin{proof}

Let
$$
 S_0=0,
 \qquad
 S_m=\sum_{\ell=0}^{m-1}\log|\rho_\ell|,
 \qquad
 B_*=\max_{P_{ij}>0}|\beta_{ij}|.
$$
 For $0\le k\le n-1$, each summand of
\eqref{eq:finite-perpetuity} satisfies
$$
 \left|\prod_{\ell=k+1}^{n-1}\rho_\ell\right|
 =\exp(S_n-S_{k+1}).
$$
Hence
\begin{align}
 |Z_n|
 &\le B_*\sum_{k=0}^{n-1}\exp(S_n-S_{k+1})\notag\\
 &\le B_*n\exp\left(
   S_n-\min_{0\le m\le n}S_m\right).
 \label{eq:perpetuity-pathwise-expanded}
\end{align}
Let
$\mathcal R_n=S_n-\min_{m\le n}S_m$.  Since
$\mathcal R_n\ge0$,
$$
 1+|Z_n|\le(1+B_*n)e^{\mathcal R_n}.
$$
The elementary estimate
$$
 1+B_*n\le(1+B_*)(n+1)
$$
gives
$$
 \log(1+B_*n)\le\log(1+B_*)+\log(n+1).
$$
Moreover,
\begin{align*}
 \mathcal R_n
 &=S_n-\min_{0\le m\le n}S_m\\
 &\le |S_n|+\left|\min_{0\le m\le n}S_m\right|\\
 &\le |S_n|+\max_{0\le m\le n}|S_m|\\
 &\le2\max_{0\le m\le n}|S_m|.
\end{align*}
Thus, for \(C_\beta=\log(1+B_*)\),
\begin{equation}\label{eq:pathwise-perpetuity-article}
 \log(1+|Z_n|)
 \le C_\beta+\log(n+1)+2\max_{m\le n}|S_m|.
\end{equation}
%

Now apply the Poisson equation specifically to the centered edge
observable $\xi(i,j)=\log|\rho_{ij}|$.  Let $u_\rho$ be its
centered Poisson solution.  Lemma~\ref{lem:martingale-coboundary} gives
$$
 S_m=M_m^\rho+u_\rho(X_0)-u_\rho(X_m),
$$
where $M_m^\rho$ has increments bounded by some $L_\rho$.  The
explicit maximal estimate \eqref{eq:martingale-maximal-expanded} gives
\begin{equation}\label{eq:rho-maximal-expanded}
 \Prob\left(
  \max_{m\le n}|S_m|>q+2\|u_\rho\|_\infty
 \right)
 \le2\exp\left(-\frac{q^2}{2nL_\rho^2}\right).
\end{equation}
Set
$$
 C_0=C_\beta+4\|u_\rho\|_\infty.
$$
If
$$
 \log(1+|Z_n|)>C_0+\log(n+1)+t,
$$
then \eqref{eq:pathwise-perpetuity-article} implies
$$
 2\max_{m\le n}|S_m|>4\|u_\rho\|_\infty+t,
$$
or equivalently
$$
 \max_{m\le n}|S_m|
 >2\|u_\rho\|_\infty+\frac t2.
$$
Using \eqref{eq:rho-maximal-expanded} with $q=t/2$, we obtain, for
every $t\ge0$,
\begin{equation}\label{eq:perpetuity-tail-explicit}
 \Prob\left(
  \log(1+|Z_n|)>C_0+\log(n+1)+t
 \right)
 \le2\exp\left(-\frac{t^2}{8nL_\rho^2}\right).
\end{equation}
This is \eqref{eq:perpetuity-tail-article}, after renaming constants.

Finally, apply the tail identity to the positive part of
$$
 W_n=\bigl(
 \log(1+|Z_n|)-C_0-\log(n+1)
 \bigr)_+.
$$
Since $W_n\ge0$,
$$
 \EE W_n=\int_0^\infty\Prob(W_n>t)\dd t.
$$
Using \eqref{eq:perpetuity-tail-explicit} gives
\begin{align*}
 \EE\log(1+|Z_n|)
 &\le C_0+\log(n+1)+\EE W_n\\
 &\le C_0+\log(n+1)
 +2\int_0^\infty
   \exp\left(-\frac{t^2}{8nL_\rho^2}\right)\dd t\\
 &\le C+C\log(n+1)+C\sqrt n.
\end{align*}
This proves \eqref{eq:perpetuity-mean-article}.  Notice that we control
the logarithm of the perpetuity; $|Z_n|$ itself may be as large as
$\exp(O(\sqrt n))$.
\end{proof}

\begin{lemma}[triangular matrix estimate]
\label{lem:triangular-matrix}
In case \textup{(iv)} of Proposition~\ref{prop:normal-forms},
\begin{equation}\label{eq:triangular-growth-article}
 \EE\log\|\widehat A^{(n)}\|
 \le C\sqrt n+C\log(n+1).
\end{equation}
\end{lemma}

\begin{proof}
Since
$$
 \frac{A_n}{D_n}=\prod_{k=0}^{n-1}\rho_k,
$$
we have
\begin{equation}\label{eq:triangular-Sn-definition}
 S_n:=\log|A_n/D_n|
 =\sum_{k=0}^{n-1}\log|\rho_k|.
\end{equation}
The determinant of the triangular product is $A_nD_n$, and
$$
 \log|A_nD_n|
 =2\log|D_n|+\log|A_n/D_n|
 =2\log|D_n|+S_n.
$$
The determinant endpoint estimate
\eqref{eq:endpoint-determinant-bound-detailed} therefore gives
\begin{equation}\label{eq:triangular-determinant-expanded}
 |2\log|D_n|+S_n|
 =|\log|A_nD_n||\le2\log L_C.
\end{equation}
In particular,
\begin{equation}\label{eq:Dn-upper-from-determinant}
 \log|D_n|\le-\frac12S_n+\log L_C.
\end{equation}

%
Factoring out $D_n$,
$$
 G_n=D_n
 \begin{pmatrix}A_n/D_n&Z_n\\0&1\end{pmatrix}.
$$
For any $x,z\in\RR$, the operator norm is bounded by the Frobenius
norm, and hence
$$
 \left\|\begin{pmatrix}x&z\\0&1\end{pmatrix}\right\|
 \le\sqrt{x^2+z^2+1}
 \le1+|x|+|z|.
$$
Because $|A_n/D_n|=e^{S_n}$, we obtain
\begin{align}
 \log\|G_n\|
 &\le\log|D_n|+log(1+e^{S_n}+|Z_n|)\notag\\
 &\le-\frac12S_n+\log L_C
       +\log(1+e^{S_n}+|Z_n|),
 \label{eq:triangular-pre-split}
\end{align}
where the second line uses \eqref{eq:Dn-upper-from-determinant}.
The elementary product inequality
$$
 1+e^{S_n}+|Z_n|
 \le(1+e^{S_n})(1+|Z_n|)
$$
gives
\begin{align*}
 \log\|G_n\|
 &\le \log L_C
 +\left[-\frac12S_n+\log(1+e^{S_n})\right]
 +\log(1+|Z_n|).
\end{align*}
The expression in brackets is symmetric in the sign of $S_n$:
$$
 -\frac12S_n+\log(1+e^{S_n})
 =\log(e^{-S_n/2}+e^{S_n/2})
 =\log(2\cosh(S_n/2))
 \le\log2+\frac12|S_n|.
$$
Therefore
\begin{equation}\label{eq:triangular-norm-expanded}
 \log\|G_n\|
 \le \log(2L_C)+\frac12|S_n|+\log(1+|Z_n|).
\end{equation}

The observable $\log|\rho_{ij}|$ is centered by
\eqref{eq:critical-drift-article}.  Lemma~\ref{lem:diffusion} applied
to \eqref{eq:triangular-Sn-definition} gives
\begin{equation}\label{eq:triangular-Sn-mean}
 \EE_P|S_n|\le C\sqrt n.
\end{equation}
Lemma~\ref{lem:critical-perpetuity} gives
\begin{equation}\label{eq:triangular-Zn-log-mean}
 \EE_P\log(1+|Z_n|)
 \le C\sqrt n+C\log(n+1).
\end{equation}
Taking expectations in \eqref{eq:triangular-norm-expanded} and using
\eqref{eq:triangular-Sn-mean}--\eqref{eq:triangular-Zn-log-mean}, we obtain
$$
 \EE_P\log\|G_n\|
 \le C\sqrt n+C\log(n+1).
$$
Finally, \eqref{eq:endpoint-norm-bound-detailed} gives
\begin{align*}
 \EE_P\log\|\widehat A^{(n)}(X)\|
 &\le\EE_P\log\|G_n\|+2\log L_C\\
 &\le C\sqrt n+C\log(n+1),
\end{align*}
which proves \eqref{eq:triangular-growth-article}.
\end{proof}

\section{Proof of the Main Theorem}

\begin{proof}
Proposition~\ref{prop:normal-forms} exhausts the possibilities.
The scalar normalization, its stationary average, and the automatic lower bound at every scale were established in
\eqref{eq:zero-gap-target-identity}--\eqref{eq:zero-gap-lower-all-scales}. The exact identity is
\begin{equation}\label{eq:Fn-zero-gap-final-identity}
 F_n(A,P)-\lambda_0
 =\frac1n\EE_P\log\|\widehat A^{(n)}(X)\|,
\end{equation}
and its left-hand side is nonnegative for every $n$ by
subadditivity.

In the compact (conformal) case, \eqref{eq:compact-growth} and
\eqref{eq:Fn-zero-gap-final-identity} give
$$
 0\le F_n(A,P)-\lambda_0\le\frac Cn.
$$
In the degenerate diagonal and simply reducible cases, \eqref{eq:diagonal-growth} gives
$$
 0\le F_n(A,P)-\lambda_0\le Cn^{-1/2}.
$$
In the degenerate triangular case, Lemma~\ref{lem:triangular-matrix} gives
$$
 0\le F_n(A,P)-\lambda_0
 \le Cn^{-1/2}+C\frac{\log(n+1)}n.
$$
Since
$$
 \log(n+1)\le C\sqrt n
 \qquad(n\ge1),
$$
the degenerate triangular bound also becomes
$$
 0\le F_n(A,P)-\lambda_0\le Cn^{-1/2}.
$$
Thus Proposition~\ref{prop:finite-time-principle} applies with
$\sigma=1/2$ in all noncompact zero-gap cases and with $\sigma=1$
in the compact case.  It yields
$$
 |\lambda_\pm(B,Q)-\lambda_0|
 \le C\left(\log\frac1\delta\right)^{-1/2},
$$
respectively the sharper exponent $1$ in the compact case, where
$\delta=\Delta((A,P),(B,Q))$.  This completes the proof.
\end{proof}

\bibliographystyle{amsalpha}
\bibliography{ref}
\end{document}